\documentclass[a4paper]{amsart}
\usepackage{amsmath, amsthm, amsfonts, amsaddr, enumerate, fullpage}
\usepackage[all]{xy}
\usepackage[inline]{enumitem}
\newtheorem{theorem}{Theorem}[section]
\newtheorem{lemma}[theorem]{Lemma}
\newtheorem{corollary}[theorem]{Corollary}
\newtheorem{conjecture}[theorem]{Conjecture}
\newtheorem{proposition}[theorem]{Proposition}
\newtheorem{algorithm}[theorem]{Algorithm}
\newtheorem{definition}[theorem]{Definition}

\theoremstyle{definition}

\theoremstyle{remark}

\numberwithin{subcase}{case}

\begin{document}
\title{On-line vertex ranking of trees}
\author{Daniel C. McDonald}
\address{Department of Mathematics, University of Illinois, 
Urbana, IL, USA}
\email{dmcdona4@illinois.edu}
\date{}
\maketitle 
\begin{abstract}
A $k$-ranking of a graph $G$ is a labeling of its vertices from $\{1,\ldots,k\}$ such that any nontrivial path whose endpoints have the same label contains a larger label.  The least $k$ for which $G$ has a $k$-ranking is the ranking number of $G$, also known as tree-depth.  Applications of rankings include VLSI design, parallel computing, and factory scheduling. The on-line ranking problem asks for an algorithm to rank the vertices of $G$ as they are presented one at a time along with all previously ranked vertices and the edges between them (so each vertex is presented as the lone unranked vertex in a partially labeled induced subgraph of $G$ whose final placement in $G$ is not specified).  The on-line ranking number of $G$ is the minimum over all such algorithms of the largest label that algorithm can be forced to use.  We give bounds on the on-line ranking number of trees in terms of maximum degree, diameter, and number of interior vertices.
\end{abstract}
\section{Introduction}
We consider a special type of proper vertex coloring using positive integers, called ``ranking.''  As with proper colorings, there exist variations on the original ranking problem.  In this paper we consider the on-line ranking problem, introduced by Tuza and Voigt in 1995 \cite{TV}.
\begin{definition}
A ranking of a finite simple graph $G$ is a function $f:V(G)\rightarrow\{1,2,\ldots\}$ with the property that if $u\neq v$ but $f(u)=f(v)$, then every $u,v$-path contains a vertex $w$ satisfying $f(w)>f(u)$ (equivalently, every path $P$ contains a unique vertex with largest label, where $f(v)$ is called the label of $v$).  A $k$-ranking of $G$ is a ranking $f:V(G)\rightarrow\{1,\ldots,k\}$.  The ranking number of a graph $G$, denoted here by $\rho(G)$ (though in the literature often as $\chi _r(G)$), is the minimum $k$ such that $G$ has a $k$-ranking.  
\end{definition}
Vertex rankings of graphs were introduced in \cite{IRV}, and results through 2003 are surveyed in \cite{J}.  Their study was motivated by applications to VLSI layout, cellular networks, Cholesky factorization, parallel processing, and computational geometry.  For example, vertex ranking models the efficient assembly of a graph from vertices, where each stage of construction consists of individual vertices being added in such a way that no component ever has more than one new vertex.  Vertex rankings are sometimes called ordered colorings, and the ranking number of a graph is trivially equal to its ``tree-depth,'' a term introduced by Ne\u{s}et\u{r}il and Ossona de Mendez in 2006 \cite{NO} in developing their theory of graph classes having bounded expansion.

The vertex ranking problem has spawned multiple variations, including list ranking \cite{M} and on-line ranking, studied here.  The on-line ranking problem is to vertex ranking as the on-line coloring problem is to ordinary vertex coloring. 
\subsection{The on-line vertex ranking problem}
The \emph{on-line vertex ranking problem} is a game between two players, Presenter and Ranker.  A class $\mathcal{G}$ of unlabeled graphs is shown to both players at the beginning of the game.  In round $1$, Presenter presents to Ranker the graph $G_1$ consisting of a single vertex $v_1$, to which Ranker assigns a positive integer label $f(v_i)$.  In round $i$ for $i>1$, Presenter extends $G_{i-1}$ to an $i$-vertex induced subgraph $G_i$ of a graph $G\in\mathcal{G}$ by presenting an unlabeled vertex $v_i$ (without specifying which copy of $G_i$ among all induced subgraphs of graphs in $\mathcal{G}$).  Ranker must then extend the ranking $f$ of $G_{i-1}$ to a ranking of $G_i$ by assigning $f(v_i)$.  

Presenter seeks to maximize the largest label assigned during the game, while Ranker seeks to minimize it. The \emph{on-line ranking number} of $\mathcal{G}$, denoted here by $\mathring{\rho}(\mathcal{G})$ (though in the literature often as $\chi ^*_r(\mathcal{G})$), is the resulting maximum assigned value under optimal play.  If Presenter can guarantee that arbitrarily high labels are used, then $\mathring{\rho}(\mathcal{G})=\infty$.  If $\mathcal{G}$ is the class of induced subgraphs of a graph $G$, then we define $\mathring{\rho}(G)=\mathring{\rho}(\mathcal{G})$.

Note that $\mathring{\rho}(\mathcal{G}')\leq\mathring{\rho}(\mathcal{G})$ if every graph in $\mathcal{G}'$ is an induced subgraph of a graph in $\mathcal{G}$, since any strategy for Ranker on $\mathcal{G}$ includes a strategy on $\mathcal{G}'$.  Also $\rho(G)\leq\mathring{\rho}(G)$ trivially.  

Several papers have been written about the on-line ranking number of graphs, including \cite{BH}, \cite{SS}, and \cite{STV}; some of the results from these papers will be mentioned later.  On-line ranking has also been looked at from an algorithmic perspective, in the sense that one seeks a fast algorithm for determining the smallest label Ranker is allowed to use on a given turn; see \cite{CL}, \cite{HIKR}, \cite{JL}, and \cite{JLW}.  Our paper is of the former variety.

A \emph{minimal ranking} of $G$ is a ranking $f$ with the property that decreasing $f$ on any set of vertices produces a non-ranking.  Let $\psi (G)$ be the largest label used in any minimal ranking of $G$.  Isaak, Jamison, and Narayan \cite{IJN} showed that the minimal rankings of $G$ are precisely the rankings produced when Ranker plays greedily, so $\mathring{\rho}(G)\leq\psi (G)$.  For the $n$-vertex path $P_n$, this yields $\mathring{\rho}(P_n)\leq\psi (P_n)=\lfloor\log_2(n +1)\rfloor +\lfloor\log_2(n +1-2^{\lfloor\log_2 n\rfloor -1})\rfloor$.  Bruoth and Hor\u{n}\'{a}k \cite{BH} gave the best known lower bound for paths $\mathring{\rho}(P_n)\geq 1.619(\log_2 n)-1$.
\subsection{Our Results}
Recall that the \emph{distance} between two vertices $u$ and $v$ in a connected graph $G$ is the number of edges in a shortest $u,v$-path in $G$.  The \emph{eccentricity} of $v$ is the greatest distance between $v$ and any other vertex in $G$.  The \emph{diameter} of $G$ is the maximum eccentricity of any vertex in $G$.

In Section \ref{sec:bounds}, we give bounds on the on-line ranking number of $T_{k,d}$, defined for $k\geq 2$ and $d\geq 0$ to be the largest tree having maximum degree $k$ and diameter $d$, i.e., the tree all of whose internal vertices have degree $k$ and all of whose leaves have eccentricity $d$.  Since the family of trees with maximum degree at most $k$ and diameter at most $d$ is precisely the set of connected induced subgraphs of $T_{k,d}$, our upper bound on $\mathring{\rho}(T_{k,d})$ also serves as an upper bound for the on-line ranking number of this class of graphs.
\begin{theorem}\label{kdbounds}
There exist positive constants $c$ and $c'$ such that if $d\geq 0$ and $k\geq 3$, then $c(k-1)^{\left\lfloor d/4\right\rfloor}\leq\mathring{\rho}(T_{k,d})\leq c'(k-1)^{\left\lfloor  d/3\right\rfloor}$.
\end{theorem}
We find it informative to compare the on-line ranking number of $T_{k,d}$ to the regular ranking number of $T_{k,d}$.  
\begin{proposition}\label{ranknumber}
For $k\geq 3$, we have $\rho(T_{k,d})=\left\lceil d/2\right\rceil +1$.
\end{proposition}
\begin{proof}
The construction for the upper bound assigns label $i+1$ to vertices at distance $i$ from the nearest leaf, with the exception of labeling one of the vertices in the central edge of $T_{k,d}$ with $(d+3)/2$ if $d$ is odd.  For the lower bound, note that choosing the unique highest ranked vertex $v$ of a tree $T$ reduces the ranking problem to individually ranking the components of $T-v$.  Thus if there exists $u\in V(T)$ such that for every $w\in V(T)$ each component of $T-u$ is isomorphic to a subtree of some component of $T-w$, then $T$ can be optimally ranked by optimally ranking each component of $T-u$ and labeling $u$ one greater than the largest label used on those components.  Letting $F_i$ denote the subforest of $T_{k,d}$ induced by the set of vertices within distance $i$ of a leaf, we conclude by induction on $i$ that for $1\leq i\leq\left\lceil d/2\right\rceil$, each component of $F_i$ is optimally ranked by the upper bound construction.
\end{proof}
Setting $n=|V(T_{k,d})|$ and using Theorem \ref{kdbounds} and Proposition \ref{ranknumber}, we see that $\mathring{\rho}(T_{k,d})=\Omega (\sqrt{n})$ while $\rho (T_{k,d})=O(\log n)$.  Thus $\mathring{\rho}$ is exponentially larger than $\rho$ on these trees.  Theorem \ref{pqbound} shows that this large separation between $\rho$ and $\mathring{\rho}$ does not hold for all trees.  Nevertheless we conjecture a general upper bound like that of Theorem \ref{kdbounds}.
\begin{conjecture}\label{treeconjecture}
There exist universal constants $a$ and $b$ satisfying $0<a<1<b$ such that $\mathring{\rho}(T)\leq b(kn)^a$ for any $n$-vertex tree $T$ with maximum degree $k$.
\end{conjecture}
In Section \ref{sec:small}, we consider the on-line ranking number of trees with few internal vertices.  Let $\mathcal{T}^{p,q}$ be the family of trees having at most $p$ internal vertices and diameter at most $q$.  The main result of that section is an upper bound on $\mathring{\rho}(\mathcal{T}^{p,q})$ for any $p$ and $q$.
\begin{theorem}\label{pqbound}
$\mathring{\rho}(\mathcal{T}^{p,q})\leq p+q+1$.
\end{theorem}
Since $q\leq p+1$, this establishes $\mathring{\rho}(\mathcal{T}^{p,q})\leq 2p+2$.  We also compute $\mathring{\rho}(\mathcal{T}^{2,3})=4$.  This extends the work of Schiermeyer, Tuza, and Voigt \cite{STV}, who characterized the families of graphs having on-line ranking number $1$, $2$, and $3$.
\newpage
\section{Strategies for Presenter and Ranker on $T_{k,d}$} 
\label{sec:bounds}
In this section, we obtain upper and lower bounds on $\mathring{\rho}(T_{k,d})$, where $T_{k,d}$ is the largest tree having maximum degree $k$ and diameter $d$.  For convenience, we let $T^*_{k,r}$ denote the tree with unique root vertex $v^*$ such that every internal vertex has $k$ children and every leaf is distance $r$ from $v^*$.  For $U\subseteq V(G)$, let $G[U]$ denote the subgraph of $G$ induced by $U$.
\subsection{A strategy for Presenter}
We first develop a tool for proving lower bounds.
\begin{theorem}\label{lowerbound}
Let $G$ be a connected graph.  Suppose for some $U\subset V(G)$ that $G-U$ has components $G^0,G^1,\ldots ,G^a$, all isomorphic to some graph $F$.  If also $U$ contains disjoint subsets $U^1,\ldots,U^a$ so that each $U^i$ consists of the internal vertices of a path joining $G^0$ and $G^i$, then $\mathring{\rho}(G)\geq\mathring{\rho}(F)+a$.
\end{theorem}
\begin{proof}
Presenter has a strategy to produce a copy of $F$ on which Ranker must use a label at least $\mathring{\rho}(F)$.  Begin by playing this strategy $a+1$ times on distinct sets of vertices.  Index the resulting copies of $F$ as $G^0,G^1,\ldots ,G^a$ so that $G^0$ is a copy whose largest label is smallest (in the labeling by Ranker) among the copies of $F$.  Present $U$ in any order to complete $G$.  

Let $m_0$ denote the largest label given to a vertex in $V(G^0)$.  For $1\leq i\leq a$, let $m_i$ denote the largest label given to a vertex in $V(G^i)\cup U^i$.  Set $H^i=G[V(G^0)\cup U^i\cup V(G^i)]$ for $1\leq i\leq a$.  Each $H^i$ is a connected subgraph of $G$, so $m_0<m_i$.  For $i\neq j$, $H^i\cup H^j$ is a connected subgraph of $G$, so $m_i\neq m_j$.  Thus the largest $m_i$ satisfies $m_i\geq m_0+a\geq\mathring{\rho}(F)+a$.
\end{proof}
\begin{figure}[ht]
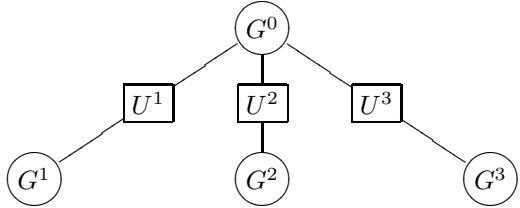

\centering
\[ \xygraph{
!{<0cm,0cm>;<3cm,0cm>:<0cm,1cm>::}
!{(1,2) }*++[o][F-:<3pt>]{G^0}="a"
!{(.5,1) }*+[F]{U^1}="b"
!{(1,1) }*+[F]{U^2}="c"
!{(1.5,1) }*+[F]{U^3}="d"
!{(0,0) }*++[o][F-:<3pt>]{G^1}="e"
!{(1,0) }*++[o][F-:<3pt>]{G^2}="f"
!{(2,0) }*++[o][F-:<3pt>]{G^3}="g"
"a"-"b"-"e" "a"-"c"-"f" "a"-"d"-"g"
} \]
\caption{The graph $G$ of Theorem \ref{lowerbound}.}
\label{lowfig}
\end{figure}
Note that $T_{k,2r}$ consists of a copy of $T^*_{k-1,r}$ and a copy of $T^*_{k-1,r-1}$ with an edge joining their roots, and $T_{k,2r+1}$ consists of two copies of $T^*_{k-1,r}$ with an edge joining their roots.  Hence $T^*_{k-1,\left\lfloor d/2\right\rfloor}$ is an induced subgraph of $T_{k,d}$, so a lower bound on $\mathring{\rho}(T^*_{k-1,\left\lfloor d/2\right\rfloor})$ also serves as a lower bound on $\mathring{\rho}(T_{k,d})$.
\begin{corollary}
If $k\geq 2$ and $r\geq 0$, then $\mathring{\rho}(T^*_{k,r})\geq k^{\left\lfloor r/2\right\rfloor}$.
\end{corollary}
\begin{proof}
Since $T^*_{k,r}$ is an induced subgraph of $T^*_{k,r+1}$, we have $\mathring{\rho}(T^*_{k,r})\leq\mathring{\rho}(T^*_{k,r+1})$, so we may assume that $r$ is even.  Set $a=k^{r/2}$, and let $U$ be the set of vertices $u_1,\ldots ,u_a$ at distance $r/2$ from $v^*$.  Define $G$ to be the subtree of $T^*_{k,r}$ obtained by deleting, for each $u_i\in U$, the descendants of all but one child of $u_i$.  Now $G-U$ consists of $a+1$ disjoint copies of $T^*_{k,r/2-1}$.  Let $G^0$ be the component rooted at $v^*$, and for $1\leq i\leq a$ let $G^i$ be the component rooted at the child of $u_i$.  Setting $U^i=\{u_i\}$ for $1\leq i\leq a$, we see that $U^i$ contains the lone vertex of the path joining $G^0$ and $G^i$.  By Theorem \ref{lowerbound}, $\mathring{\rho}(T^*_{k,r})\geq\mathring{\rho}(G)\geq a$.
\end{proof}
\begin{corollary}
If $k\geq 3$ and $d\geq 0$, then $\mathring{\rho}(T_{k,d})\geq (k-1)^{\left\lfloor d/4\right\rfloor}$.
\end{corollary}
We finish this subsection with a comment on Conjecture \ref{treeconjecture}.  Subdivide each edge of the star $K_{1,a}$ to get a $(2a+1)$-vertex tree $G$.  Letting $G^0,G^1,\ldots,G^a$ correspond to the vertices of the unique maximum independent set of $G$, Theorem \ref{lowerbound} yields $\mathring{\rho}(G)\geq a+1>|V(G)|/2$.  Thus Conjecture \ref{treeconjecture} cannot be strengthened to the statement ``There exist universal constants $a$ and $b$ satisfying $0<a<1<b$ such that $\mathring{\rho}(T)\leq bn^a$ for any tree $n$-vertex tree $T$.''
\subsection{A strategy for Ranker}
We now exhibit a strategy for Ranker to establish an upper bound on $\mathring{\rho}(T_{k,d})$.  In Section \ref{sec:small} we shall see $\mathring{\rho}(T_{k,0})=1$, $\mathring{\rho}(T_{k,1})=2$, $\mathring{\rho}(T_{k,2})=3$, $\mathring{\rho}(T_{k,3})=4$, $\mathring{\rho}(T_{k,4})\leq k+6$, and $\mathring{\rho}(T_{k,5})\leq 2k+6$, so here we only consider $d\geq 6$.  In specifying a strategy for Ranker on $T_{k,d}$, we will give a procedure for ranking the presented vertex $v$ based solely on the component containing $v$ in the graph presented so far.
\begin{definition}
Let $T(v)$ denote the component containing $v$ when $v$ is presented.  Given two sets $A$ and $B$ of labels, not necessarily disjoint, let $T_B(v)$ be the largest subtree of $T(v)$ containing $v$ all of whose other vertices are labeled from $B$.  Should it exist, let $f^A_B(v)$ denote the smallest element of $A$ that would complete a ranking of $T_B(v)$.
\end{definition}  
The following lemmas analyze when $f^A_B(v)$ exists and, if it does exist, when $f^A_B(v)$ provides a valid label that Ranker can give $v$.
\begin{lemma}\label{exist}
Suppose that each vertex $u\in V(T_B(v))$ labeled from $A$ was given label $f^A_B(u)$ when it arrived.  If $\min A>\max (B-A)$, and every component of $T_B(v)-v$ lacks some label in $A$, then $f^A_B(v)$ exists.
\end{lemma}
\begin{proof}
Let $A=\{a_1,\ldots,a_m\}$, with $a_1<\ldots <a_m$.  For a component $T$ of $T_B(v)-v$ having $q$ distinct labels from $A$, we claim that the largest label used on $T$ is $a_q$.  Each vertex $u\in V(T_B(v))$ labeled from $A$ was given label $f^A_B(u)$ when it arrived, with $\min A>\max (B-A)$, so if $f^A_B(u)=a_i$ then either $i=1$ or $a_{i-1}$ was already used in $T_B(u)$ (since otherwise $a_{i-1}$ would complete a ranking).  Hence all used labels are less than all missing labels in $A$.  Since every component of $T_B(v)-v$ lacks some label in $A$, we thus have $a_q<a_m$.  Therefore $a_m$ is a valid label for $v$ in $T_B(v)$ because the largest label on any path through $v$ would be used only at $v$.  Hence $f^A_B(v)$ exists.
\end{proof}
\begin{lemma}\label{valid}
Suppose that $f^A_B(v)$ exists.  If $T(v)=T_B(v)$ or if all vertices of $T(v)-V(T_B(v))$ having a neighbor in $T_B(v)$ are in the same component of $T(v)-v$ and have labels larger than $\max (A\cup B)$, then setting $f(v)=f^A_B(v)$ is a valid move by Ranker.
\end{lemma}
\begin{proof}
Set $f(v)=f^A_B(v)$.  Let $P$ be an $x,y$-path in $T(v)$ such that $x\neq y$, $f(x)=f(y)=\ell$, and $v\in V(P)$.  We show that $P$ has an internal vertex $z$ satisfying $f(z)>\ell$.  Since $f^A_B(v)$ completes a ranking of $T_B(v)$, we may assume that $T(v)\neq T_B(v)$ and $P$ contains some vertex outside $T_B(v)$.  By hypothesis all such vertices having a neighbor in $T_B(v)$ are in the same component of $T(v)-v$, so we may assume $x\in V(T(v))-V(T_B(v))$ and $y\in V(T_B(v))$.  

Since $v$ is labeled from $A$ and $T_B(v)-v$ is labeled from $B$ with $y\in V(T_B(v))$, we have $\ell\in A\cup B$.  By hypothesis all vertices of $T(v)-V(T_B(v))$ having neighbors in $T_B(v)$ have labels larger than $\max (A\cup B)$, so $x$ has no neighbor in $T_B(v)$.  Hence $P$ contains some internal vertex $z$ outside $T_B(v)$ with a neighbor in $T_B(v)$.  By hypothesis, $f(z)>\max (A\cup B)\geq\ell$.
\end{proof}
Set $j=\left\lfloor d/3\right\rfloor$.  Break the labels from $1$ to $3|V(T^*_{k-1,j})|$ into three segments, with $X$ consisting of the lowest $|V(T^*_{k-1,j-1})|$ labels, $Y$ the next $|V(T^*_{k-1,j})|-|V(T^*_{k-1,j-1})|$ labels, and $Z$ the remaining high labels.  For $k\geq 3$, we give Ranker a strategy in the on-line ranking game on $T_{k,d}$ that uses labels from $X\cup Y\cup Z$.  Since $\mathring{\rho}(T_{k,d})\leq 3|V(T^*_{k-1,j})|=3((k-1)^j+\sum^{j-1}_{i=0}(k-1)^i)<6(k-1)^j$, this establishes the following.
\begin{theorem}
If $d\geq 0$ and $k\geq 3$, then $\mathring{\rho}(T_{k,d})\leq 6(k-1)^{\left\lfloor  d/3\right\rfloor}$.
\end{theorem}
The goal of our strategy for Ranker is to label from $X\cup Y$ many vertices that lie within distance $j-1$ of a leaf, reserving $Z$ for a small number of middle vertices.  
\begin{algorithm}\label{Rankcomplete}
Compute $f(v)$ according to the following table.
\end{algorithm}
\begin{tabular}{|c|p{12cm}|} \hline
\textbf{Value of f(v)} & \textbf{Conditions} \\ \hline
(I) $f^X_X(v)$ &
\begin{enumerate*}[label=(\arabic*),itemjoin={\newline}]
\item $T_X(v)$ is isomorphic to a subgraph of $T^*_{k-1,j-1}$, and
\item either $T_X(v)=T(v)$ or there exists a vertex $u$ in $T(v)$ labeled from $Y$ such that $T_X(v)$ is the component of $T(v)-u$ containing $v$.
\end{enumerate*} \\ \hline
(II) $f^Y_{X\cup Y}(v)$ & \begin{enumerate*}[label=(\arabic*),itemjoin={\newline}]
\item The eccentricity of $v$ in $T(v)$ is at least $d-j$, and
\item there exists no vertex $u$ in $T(v)$ labeled from $Y$ such that $T_X(v)$ is the component of $T(v)-u$ containing $v$.
\end{enumerate*} \\ \hline
(III) $f^Z_{X\cup Y\cup Z}(v)$ & \begin{enumerate*}[label=(\arabic*),itemjoin={\newline}]
\item The eccentricity of $v$ in $T(v)$ is less than $d-j$, and 
\item either $T_X(v)$ is not isomorphic to a subgraph of $T^*_{k-1,j-1}$ or $T_X(v)\neq T(v)$.
\end{enumerate*} \\ \hline
\end{tabular}
\newpage
Before we go any further, we need to show that Algorithm \ref{Rankcomplete} is, in fact, an algorithm.  Note that $d-j\geq 2j$.
\begin{proposition}
When playing the on-line ranking game on $T_{k,d}$, each presented vertex $v$ satisfies the conditions of exactly one of the three cases.
\end{proposition}
\begin{proof}
If the eccentricity of $v$ in $T(v)$ is less than $d-j$, then Case II does not apply.  If furthermore $T_X(v)$ is isomorphic to a subgraph of $T^*_{k-1,j-1}$ and $T_X(v)=T(v)$, then Case I applies but Case III does not.  Otherwise, Case III applies, but Case I does not since if $T_X(v)$ is isomorphic to a subgraph of $T^*_{k-1,j-1}$, then $T_X(v)\neq T(v)$ and a vertex $u$ in $T(v)$ such that $T_X(v)$ is the component of $T(v)-u$ containing $v$ would have eccentricity at most $\max\{d-j-2,2j-1\}$, which is less than $d-j$, precluding $u$ from being labeled from $Y$.

If the eccentricity of $v$ in $T(v)$ is at least $d-j$, then Case III does not apply.  If furthermore $T_X(v)$ is isomorphic to a subgraph of $T^*_{k-1,j-1}$, then the eccentricity of $v$ in $T_X(v)$ is at most $2j-2$, so $T_X(v)\neq T(v)$ since $2j-2<d-j$.  Thus Case I only applies if $T_X(v)$ is isomorphic to a subgraph of $T^*_{k-1,j-1}$ and there exists a vertex $u$ in $T(v)$ labeled from $Y$ such that $T_X(v)$ is the component of $T(v)-u$ containing $v$.

If there does exist a vertex $u$ in $T(v)$ labeled from $Y$ such that $T_X(v)$ is the component of $T(v)-u$ containing $v$, then $u$ had eccentricity at least $d-j$ in $T(u)$, so $T_X(v)$ is isomorphic to a subgraph of $T^*_{k-1,j-1}$ since $T_{k,d}$ has diameter $d$.  Hence Case I applies.  If there exists no vertex $u$ in $T(v)$ labeled from $Y$ such that $T_X(v)$ is the component of $T(v)-u$ containing $v$, then Case II applies.  
\end{proof}
We now show that Algorithm \ref{Rankcomplete} produces a valid label in each of the three cases.  Assume that the algorithm has assigned valid labels before the presentation of $v$.  Note that for $(A,B)\in\{(X,X),(Y,X\cup Y),(Z,X\cup Y\cup Z)\}$, each vertex $u\in V(T_B(v))$ labeled from $A$ was given label $f^A_B(u)$ when it arrived, and $\min A>\max (B-A)$.  Hence by Lemma \ref{exist}, $f^A_B(v)$ exists if every component of $T_B(v)-v$ lacks some label in $A$.  
\begin{proposition}
In Case I, $f^X_X(v)$ exists, and setting $f(v)=f^X_X(v)$ is a valid move for Ranker.
\end{proposition}
\begin{proof}
Note that $f^X_X(v)$ exists by Lemma \ref{exist} because $|V(T_X(v))|\leq |X|$.  Furthermore, $f^X_X(v)$ provides a valid label for $v$ by Lemma \ref{valid} because either $T_X(v)=T(v)$ or there exists a vertex $u$ in $T(v)$ such that $f(u)>\max X$ and $T_X(v)$ is a component of $T(v)-u$, making $u$ the only vertex outside $T_X(v)$ neighboring a vertex inside $T_X(v)$.
\end{proof}
If $y$ satisfies the conditions of Case II, then let $H(y)$ be the component of $T(y)-y$ having greatest diameter.
\begin{lemma}\label{separate}
If $y$ is labeled from $Y$, then each vertex separated from $H(y)$ by $y$ (at any point in the game) is labeled from $X$.
\end{lemma}
\begin{proof}
The eccentricity of $y$ in $T(y)$ is at least $d-j$, so $H(y)$ has diameter at least $d-j-1$.  This forces each other component of $T(y)-y$ to be isomorphic to a subtree of $T^*_{k-1,j-1}$.  Any vertex $r$ of such a component is labeled from $X$, since $T(r)$ was isomorphic to a subgraph of $T^*_{k-1,j-1}$, implying $T_X(r)=T(r)$.  Furthermore, any subsequently presented vertex $s$ satisfying $y\in V(T(s))$ that is separated from $H(y)$ by $y$ is labeled from $X$, since $T_X(s)$ is isomorphic to a subgraph of $T^*_{k-1,j-1}$ and is the component of $T(s)-y$ containing $s$.
\end{proof}
\begin{lemma}\label{yyy}
Every path in $T_{X\cup Y}(v)$ contains at most two vertices labeled from $Y$ (including possibly $v$).  
\end{lemma}
\begin{proof}
Let $y$, $y'$, and $y''$ be distinct vertices in $T_{X\cup Y}(v)$ labeled from $Y$ (one could possibly be $v$).  Since $y'$ and $y''$ are labeled from $Y$, neither is separated from $H(y)$ by $y$, by Lemma \ref{separate}.  If $u$ is the neighbor of $y$ in $H(y)$, then the edge $uy$ must be part of any path containing $y$ and at least one of $y'$ or $y''$.  Hence edge-disjoint $y',y$- and $y,y''$-paths do not exist, so no path contains $y$ between $y'$ and $y''$.  By symmetry, no path contains each of $y$, $y'$, and $y''$.
\end{proof}
\begin{lemma}\label{zyy}
If $T(v)$ contains a vertex labeled from $Y$ (possibly $v$), then $T(v)$ contains a vertex labeled from $Z$, and no path in $T(v)$ contains a vertex labeled from $Z$ and multiple vertices of $T_{X\cup Y}(v)$ labeled from $Y$.
\end{lemma}
\begin{proof}
For the first claim, let $y$ be the first vertex in $T(v)$ labeled from $Y$.  The diameter of $H(y)$ is greater than the diameter of $T^*_{k-1,j-1}$ because $d-j-1>2j-2$, so some vertex $r\in V(H(y))$ violated the first condition of Case I when presented and was thus not labeled from $X$.  Since $r$ was presented before $y$, it is labeled from $Z$.

For the second claim, let $z$ be a vertex of $T(v)$ labeled from $Z$, and $y'$ and $y''$ be distinct vertices of $T_{X\cup Y}(v)$ labeled from $Y$.  If $u$ is the neighbor of $z$ in the direction of $v$, then the edge $uz$ must be part of any path containing $z$ and at least one of $y'$ or $y''$.  Hence edge-disjoint $y',z$- and $z,y''$-paths do not exist, so no path can contain $z$ between $y'$ and $y''$.

By Lemma \ref{separate}, any vertex separated from $H(y')$ by $y'$ is labeled from $X$, so $y''$ is not separated from $z$ by $y'$.  Similarly, $y'$ is not separated from $z$ by $y''$.  Thus no path can contain each of $z$, $y'$, and $y''$.
\end{proof}
\begin{figure}[ht]
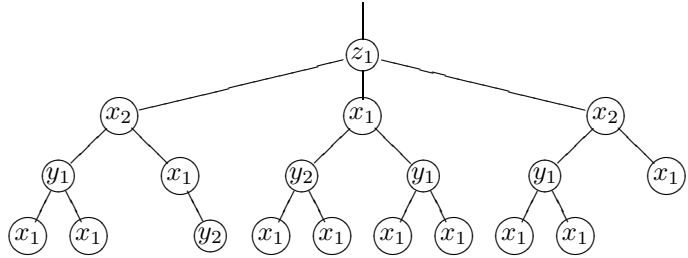

\centering
\[ \xygraph{
!{<0cm,0cm>;<1.6cm,0cm>:<0cm,.8cm>::}
!{(2.75,4) }{}="w"
!{(2.75,3) }*+[o][F-:<3pt>]{z_1}="z"
!{(.75,2) }*+[o][F-:<3pt>]{x_2}="a"
!{(2.75,2) }*+[o][F-:<3pt>]{x_1}="b"
!{(.25,1) }*+[o][F-:<3pt>]{y_1}="c"
!{(1.25,1) }*+[o][F-:<3pt>]{x_1}="d"
!{(2.25,1) }*+[o][F-:<3pt>]{y_2}="e"
!{(3.25,1) }*+[o][F-:<3pt>]{y_1}="f"
!{(0,0) }*+[o][F-:<3pt>]{x_1}="g"
!{(.5,0) }*+[o][F-:<3pt>]{x_1}="h"
!{(1.5,0) }*+[o][F-:<3pt>]{y_2}="j"
!{(2,0) }*+[o][F-:<3pt>]{x_1}="k"
!{(2.5,0) }*+[o][F-:<3pt>]{x_1}="l"
!{(3,0) }*+[o][F-:<3pt>]{x_1}="m"
!{(3.5,0) }*+[o][F-:<3pt>]{x_1}="n"
!{(4.75,2) }*+[o][F-:<3pt>]{x_2}="o"
!{(4.25,1) }*+[o][F-:<3pt>]{y_1}="p"
!{(5.25,1) }*+[o][F-:<3pt>]{x_1}="q"
!{(4,0) }*+[o][F-:<3pt>]{x_1}="r"
!{(4.5,0) }*+[o][F-:<3pt>]{x_1}="s"
"b"-"z"-"a" "w"-"z"-"o"
"g"-"c"-"a"
"b"-"f"-"n" "a"-"d"-"j"
"c"-"h" "d"
"b"-"e"-"k"
"e"-"l" "f"-"m"
"r"-"p"-"o"-"q"
"p"-"s" "q"
} \]
\caption{A labeling of $T_{k,d}$ for vertices with high eccentricity ($x_i\in X,y_i\in Y,z_i\in Z$).}
\end{figure}
\begin{proposition}
In Case II, $f^Y_{X\cup Y}(v)$ exists, and setting $f(v)=f^Y_{X\cup Y}(v)$ is a valid move for Ranker.
\end{proposition}
\begin{proof}
Let $S$ be the set consisting of $v$ and every vertex in $T_{X\cup Y}(v)$ labeled from $Y$.  By Lemma \ref{yyy}, the elements of $S$ are only separated by vertices labeled from $X$, so the smallest subtree $T$ of $T_{X\cup Y}(v)$ containing all of $S$ has all its internal vertices labeled from $X$.  Therefore the set of internal vertices of $T$ induces a tree $T'$ isomorphic to a subtree of $T^*_{k-1,j-1}$.  By Lemma \ref{zyy}, some vertex not labeled from $Y$ neighbors a vertex in $T'$ if $T'\neq\emptyset$, or else some path in $T(v)$ contains a vertex labeled from $Z$ and multiple vertices of $T_{X\cup Y}(v)$ labeled from $Y$.  Thus $|S|=|V(T)|-|V(T')|\leq |V(T^*_{k-1,j})|-|V(T^*_{k-1,j-1})|=|Y|$, so $f^Y_{X\cup Y}(v)$ exists by Lemma \ref{exist}.  

Finally, the only vertices outside $T_{X\cup Y}(v)$ that neighbor a vertex inside $T_{X\cup Y}(v)$ are in $H(y)$ and labeled from $Z$.  Hence $f^Y_{X\cup Y}(v)$ provides a valid label for $v$, by Lemma \ref{valid}.
\end{proof}
\begin{lemma}\label{z}
If $v$ is assigned a label $m\in Z$ previously unused in $T(v)$, then $v$ is a leaf of some subtree of $T_{X\cup Z}(v)$ containing every label in $Z$ smaller than $m$.
\end{lemma}
\begin{proof}
By Lemma \ref{separate}, two vertices labeled from $Z$ are never separated by a vertex labeled from $Y$, so all vertices in $T(v)-v$ labeled from $Z$ lie in $T_{X\cup Z}(v)$.  We use induction on $m$, with the base case $m=\min Z$ being trivial.  If $m>\min Z$, let $u$ be the first vertex in $T_{X\cup Z}(v)$ labeled with $m-1$.  Since $u$ arrived as a leaf of some subtree containing every label in $Z$ smaller than $m-1$, adding to that tree the $u,v$-path through $T_{X\cup Z}(v)$ yields the desired tree.
\end{proof}
\begin{lemma}\label{largest}
The largest subtree $T$ of $T_{k,d}$ having diameter $d-j-1$ has at most $2|V(T^*_{k-1,j})|$ vertices.
\end{lemma}
\begin{proof}
Let $u_1u_2$ be the central edge of $T$ if $d-j-1$ is odd and any edge containing the central vertex of $T$ if $d-j-1$ is even.  Deleting $u_1u_2$ from $T$ then leaves two trees $T_1$ and $T_2$ containing $u_1$ and $u_2$, respectively, with $u_i$ having degree at most $k-1$ and eccentricity at most $\left\lfloor (d-j-1)/2\right\rfloor$ in $T_i$.  Thus each $T_i$ is isomorphic to a subtree of $T^*_{k-1,j}$, since $\left\lfloor (d-j-1)/2\right\rfloor\leq j$ for $j=\left\lfloor d/3\right\rfloor$.  Hence $|V(T)|=|V(T_1)|+|V(T_2)|\leq 2|V(T^*_{k-1,j})|$.
\end{proof}
\begin{proposition}
In Case III, $f^Z_{X\cup Y\cup Z}(v)$ exists, and setting $f(v)=f^Z_{X\cup Y\cup Z}(v)$ is a valid move for Ranker.
\end{proposition}
\begin{proof}
Note that $T_{X\cup Y\cup Z}(v)=T(v)$, so if $f^Z_{X\cup Y\cup Z}(v)$ exists, then by Lemma \ref{valid} it is a valid label for $v$.  If $T(v)$ uses at most $2|V(T^*_{k-1,j})|$ labels from $Z$, then by Lemma \ref{exist} $f^Z_{X\cup Y\cup Z}(v)$ exists, since $|Z|=2|V(T^*_{k-1,j})|$.  By Lemma \ref{z} and the first condition of Case III, the number of labels from $Z$ used in $T(v)$ is at most the number of times a vertex $u$ in $T_{X\cup Z}(v)$ was presented as a leaf of $T_{X\cup Y}(u)$ having eccentricity less than $d-j$ in $T_{X\cup Y}(u)$.  Since any leaf added adjacent to a vertex having eccentricity at least $d-j$ will itself have eccentricity at least $d-j$, it suffices to show that growing a subtree of $T_{k,d}$ by iteratively adding one leaf $2|V(T^*_{k-1,j})|$ times eventually forces some new leaf to have eccentricity at least $d-j$ at the time of its insertion.  Since any leaf whose insertion raises the diameter of the tree has eccentricity equal to the higher diameter, this statement follows from Lemma \ref{largest}.
\end{proof}
\section{Trees with few internal vertices} 
\label{sec:small}
Recall that $\mathcal{T}^{p,q}$ is the family of trees having at most $p$ internal vertices and diameter at most $q$.  We first exhibit a strategy for Ranker on $\mathcal{T}^{p,q}$ that uses no label larger than $p+q+1$.  We can improve this bound for the class of double stars by proving $\mathring{\rho}(\mathcal{T}^{2,3})=4$ (since every tree with diameter $3$ has exactly two internal vertices, $\mathcal{T}^{2,3}$ is the family of trees with diameter $3$).  This extends the work of Schiermeyer, Tuza, and Voigt \cite{STV}, who characterized the families of graphs with on-line ranking number $1$, $2$, and $3$.
\subsection{Upper bound on $\mathring{\mathcal{T}}^{p,q}$}
During the on-line ranking game on $\mathcal{T}^{p,q}$, let $S$ be the component of the current graph containing the unlabeled presented vertex $v$.  We give Ranker a procedure for ranking $v$ based solely on $S$ and the labels given to the other vertices of $S$.
\begin{algorithm}\label{Ranksmall}
If $v$ is the only vertex in $S$, let $f(v)=q+1$.  If $v$ is not the only vertex in $S$, then let $m$ denote the largest label already used on $S$.  If there exists a label smaller than $m$ that completes a ranking when assigned to $v$, give $v$ the largest such label.  Otherwise, let $f(v)=m+1$.
\end{algorithm}
\begin{lemma}\label{Leaf}
If $v$ arrives as a leaf of a nontrivial component $S$ whose highest ranked vertex has label $m$, then Algorithm \ref{Ranksmall} will assign $v$ a label smaller than $m$.
\end{lemma}
\begin{proof}
Suppose that Algorithm \ref{Ranksmall} assigns $f(v)=m+1$.  Let $v_0=v$.  We now select vertices $v_1,\ldots,v_j$ from $S$ such that $v_0,v_1,\ldots,v_j$ in order form a path $P$ and $v_j$ arrived as an isolated vertex.  For $i\geq 0$, let $v_{i+1}$ be a vertex with the least label among all vertices that were adjacent to $v_i$ when $v_i$ was presented, unless $v_i$ arrived as an isolated vertex, in which case set $j=i$.  Since $S$ is finite, the process must end with some vertex $v_j$.  Since $v_i$ was presented as a neighbor of $v_{i+1}$, $P$ is a path.

Note that Algorithm \ref{Ranksmall} assigns $f(u)=a\neq q+1$ only if $u$ arrives as a neighbor of a vertex $w$ such that $f(w)\leq a+1$.  Since $f(v_1)=1$ (otherwise $f(v_0)=f(v_1)-1<m$), we must have $f(v_i)\leq i$ for $1\leq i<j$.  Also, $f(v_j)=q+1$ because $v_j$ arrived as an isolated vertex.  Since $v_j$ was chosen as the neighbor with the least label when $v_{j-1}$ arrived, $f(u)>q$ for any such neighbor $u$.  Hence $f(v_{j-1})\geq q$.  Therefore $j-1\geq q$, which gives $P$ length $q+1$, contradicting $S$ having diameter at most $q$.
\end{proof}
\begin{theorem}\label{RanksmallThm}
Algorithm \ref{Ranksmall} uses no label larger than $p+q+1$.
\end{theorem}
\begin{proof}
By Lemma \ref{Leaf}, the only way for a new largest label greater than $q+1$ to be used on $S$ is for the unlabeled vertex to arrive as an internal vertex.  Only the $p$ internal vertices of an element of $\mathcal{T}^{p,q}$ can be presented as such, and each time a new largest label is used it increases the largest used value by $1$, so the largest label that could be used on one of them would be $p+q+1$.
\end{proof}
\subsection{Double stars}
For any forest $F$, Schiermeyer, Tuza, and Voigt \cite{STV} proved $\mathring{\rho}(F)=1$ if and only if $F$ has no edges, $\mathring{\rho}(F)=2$ if and only if $F$ has an edge but no component with more than one edge, and $\mathring{\rho}(F)=3$ if and only if $F$ is a star forest with maximum degree at least $2$ or $F$ is a linear forest whose largest component is $P_4$.  Since $P_4$ is the only member of $\mathcal{T}^{2,3}$ having on-line ranking number less than $4$, proving $\mathring{\rho}(\mathcal{T}^{2,3})=4$ only requires a strategy for Ranker, and our result implies $\mathring{\rho}(T)=4$ for any $T\in\mathcal{T}^{2,3}-\{P_4\}$.  We now make some observations about the on-line ranking game on $\mathcal{T}^{2,3}$ before giving a strategy for Ranker.

When a vertex $u$ is presented, let $G(u)$ be the graph at that time, and let $T(u)$ be the component of $G(u)$ containing $u$.  When the first edge(s) appear, the presented vertex $v$ is the center of a star; thus $T(v)$ is a star, while $G(v)$ may include isolated vertices in addition to $T(v)$.  Let $v'$ be the first vertex to complete a path of length $3$.  The graph $G(v')$ is connected and has two internal vertices, properties that remain true as subsequent vertices are presented.  Let $T$ be the final tree.

Consider the round when a vertex $u$ is presented.  If $u$ is presented after $v'$, or $u=v'$ and $u$ is a leaf of $T(u)$, then $G(u)=T(u)$, and $u$ must be a leaf in $T$.  If $u$ is presented after $v$ but before $v'$, then either $T(u)=u$ or $T(u)$ is a star not centered at $u$.  If additionally $G(u)$ is disconnected, then $u$ must wind up as a leaf in $T$, since $T$ has diameter $3$.  Call $u$ a \emph{forced leaf} in this case, the case that $u$ is presented after $v'$, or the case that $u=v'$ and $u$ is presented as a leaf of $T(u)$.  Otherwise, if $u$ is presented after $v$ but before $v'$, then $u$ is a leaf of $T(u)$, and say that $u$ is \emph{undetermined} (since $u$ may or may not wind up as a leaf in $T$).  Also call $v$ undetermined, as well as $v'$ if $v'$ is not a forced leaf.
\begin{algorithm}\label{Rankdoublestar}
Give label $3$ to the first vertex presented, label $2$ to any subsequent vertex presented before $v$, and label $1$ to any forced leaf.  The rest of the algorithm specifies how to rank the undetermined vertices in terms of the labeling of $G(v)$.

If $G(v)=P_2$, then give label $4$ to $v$ and label $2$ to any subsequent undetermined vertex.  If $G(v)$ has more than one edge (disconnected or not), and $v$ is adjacent to the vertex labeled $3$, then give label $4$ to $v$ and label $3$ to any subsequent undetermined vertex.

If neither of the previous cases hold, then $G(v)$ is disconnected, and $v$ and $v'$ are the only undetermined vertices.  If $G(v)$ has exactly one edge, and $v$ is adjacent to the vertex labeled $3$, then give label $2$ to $v$ and label $4$ to $v'$.  In the remaining case, $v$ is not adjacent to the vertex labeled $3$; give label $3$ to $v$ and label $4$ to $v'$.
\end{algorithm}
\begin{figure}[ht]
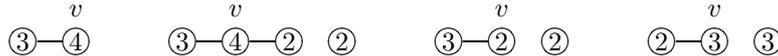

\centering
\[ \xygraph{
!{<0cm,0cm>;<0.7cm,0cm>:<0cm,2cm>::}
!{(0,0) }*+[o][F-:<3pt>]{3}="a"
!{(1,0) }*+[o][F-:<3pt>]{4}="b"
!{(1,.2) }*{v}="v1"
!{(3,0) }*+[o][F-:<3pt>]{3}="c"
!{(4,0) }*+[o][F-:<3pt>]{4}="d"
!{(4,.2) }*{v}="v2"
!{(5,0) }*+[o][F-:<3pt>]{2}="e"
!{(6,0) }*+[o][F-:<3pt>]{2}="f"
!{(8,0) }*+[o][F-:<3pt>]{3}="g"
!{(9,0) }*+[o][F-:<3pt>]{2}="h"
!{(9,.2) }*{v}="v3"
!{(10,0) }*+[o][F-:<3pt>]{2}="i"
!{(12,0) }*+[o][F-:<3pt>]{2}="j"
!{(13,0) }*+[o][F-:<3pt>]{3}="k"
!{(13,.2) }*{v}="v4"
!{(14,0) }*+[o][F-:<3pt>]{3}="l"
"a"-"b"
"c"-"d"-"e"
"g"-"h"
"j"-"k"
} \]
\caption{Possibilities for $G(v)$.}
\end{figure}
\begin{proposition}
$\mathring{\rho}(\mathcal{T}^{2,3})=4$.
\end{proposition}
\begin{proof}
Because $P_4$ is the only tree with exactly two internal vertices having on-line ranking number at most 3, we need only to verify that Algorithm \ref{Rankdoublestar} is a valid strategy for Ranker.  

If $G(v)=P_2$, then every vertex labeled $1$ is a leaf, and the only label besides $1$ that can be used more than once is $2$.  Any two vertices labeled $2$ must be separated by one of the first two vertices presented, each of which receives a higher label.

If $G(v)$ has more than one edge, and $v$ is added adjacent to the vertex labeled $3$, then every vertex labeled $1$ is a leaf, and the only vertex labeled $4$ is $v$, which is an internal vertex.  If the other internal vertex is labeled $3$, then each leaf adjacent to it is labeled $1$ or $2$.  Any two vertices labeled $3$ must be separated from each other by $v$, which is labeled $4$, and any two vertices labeled $2$ must be separated from each other by an internal vertex, which is labeled either $3$ or $4$.  If the internal vertex besides $v$ is labeled $2$, then each adjacent leaf must be labeled $1$.  Any two vertices with the same label of $2$ or $3$ would have to be separated from each other by $v$, which is labeled $4$. 

If $G(v)$ has exactly one edge but more than two vertices, and $v$ is adjacent to the vertex labeled $3$, then any vertex labeled $1$ will be a leaf, only the first vertex presented will be labeled $3$, and any two vertices labeled $2$ will be separated from each other by $v'$, which is the only vertex labeled $4$.

If $G(v)$ has more than two vertices, and $v$ is not adjacent to the vertex labeled $3$, then any vertex labeled $1$ will be a leaf, and any two vertices with the same label of $2$ or $3$ will be separated from each other by $v'$, which is the only vertex labeled $4$.
\end{proof}
\section*{Acknowledgments}
Special thanks to Prof. Doug West for his helpful guidance and editing advice.  This work was supported by National Science Foundation grant DMS 08-38434 EMSW21-MCTP: Research Experience for Graduate Students.

\end{document}